\documentclass[a4paper, 10pt]{amsart}

\usepackage{amssymb,amsthm}

\newcommand{\rr}{\mathbb{R}}
\newcommand{\nn}{\mathbb{N}}
\renewcommand{\d}{{\rm d}}
\renewcommand{\.}{\cdot}
\newcommand{\e}{\epsilon}
\newcommand{\s}{\sigma}
\newcommand{\test}{\varphi}
\renewcommand{\div}{{\rm div}}

\newtheorem{theorem}{Theorem}
\newtheorem{proposition}[theorem]{Proposition}
\newtheorem{lemma}[theorem]{Lemma}

\begin{document}

\title[Effects of Bohm's potential on a macroscopic system]{On the effects of Bohm's potential on a macroscopic system of self-interacting particles}
\author{Oliver Tse}
\address{Department of Mathematics, University of Kaiserslautern, 67663 Kaiserslautern, Germany}
\email{tse@mathematik.uni-kl.de}
\date{}
\begin{abstract}
 We consider an instationary macroscopic system of self-interacting particles with an additional potential, the so called Bohm's potential. We study the existence of non-negative global solutions to the ($4$-th order) system of equations and allude the differences to results obtained for classical models. The problem is considered on a bounded domain up to three space dimension, subject to initial and Neumann boundary condition for the particle density, and Dirichlet boundary condition for the self-interacting potential. Moreover, the initial datum is only assumed to be non-negative and to satisfy a weak integrability condition.
\end{abstract}
\maketitle

\section{Introduction}

Consider a macroscopic system of self-interacting particles with \emph{Bohm's potential}, which describe the evolution of the \emph{normalized density} $n\ge 0$,
\begin{subequations}\label{eq::system}
\begin{equation}
 n_t = \div\big(n\nabla F\big)\;\;\mbox{in}\;\Omega,\quad 
 n\partial_\nu F = 0 \;\;\mbox{on}\;\Gamma,
\end{equation}
with the \emph{quasi fermi-level} $F$ given by
\begin{equation}
 F = -\e^2\frac{\Delta\sqrt{n}}{\sqrt{n}} + \log{n} - \s\Phi \;\;\mbox{in}\;\Omega,\quad
 \partial_\nu\sqrt{n} = 0 \;\;\mbox{on}\;\Gamma,
\end{equation}
and the \emph{potential} $\Phi$ due to self-interaction in a particle system,
\begin{equation}
 -\Delta \Phi = n \;\;\mbox{in}\;\Omega,\quad 
 \Phi = 0 \;\;\mbox{on}\;\Gamma,
\end{equation}
\end{subequations}
supplemented with the initial condition $n(0,\.)=n_0\ge 0$, where $\nu$ is the outer normal to the convex, bounded domain $\Omega\subset\rr^d$, $d\le 3$ with Lipschitz boundary $\Gamma$, $\e>0$ is the scaled Planck constant and $|\s|\in[0,\infty)$ is the mass of the system of self-interacting particles, where sign($\s$) dictates the nature of the interaction involved. In this case, positive mass $\s>0$ would indicate the presence of self-attraction, while negative mass $\s<0$ indicates self-repulsion. Notice that the presence of the Bohm potential $\Delta\sqrt{n}/\sqrt{n}$, which is a non-local second-order term, leads to a fourth-order evolution equation for the normalized density $n$, given by
\begin{subequations}\label{eq::4th_order}
\begin{equation}
 n_t = \div\left(n\nabla \Big(-\e^2\frac{\Delta\sqrt{n}}{\sqrt{n}} + \log{n} - \s\Phi\Big)\right)\;\;\mbox{in}\;\Omega,
\end{equation}
with the natural boundary conditions
\begin{equation}
 n\partial_\nu \big(-\e^2\frac{\Delta\sqrt{n}}{\sqrt{n}} + \log{n} - \s\Phi\big) = 0 \;\;\mbox{on}\;\Gamma,\quad\mbox{and}\quad \partial_\nu\sqrt{n} = 0 \;\;\mbox{on}\;\Gamma.
\end{equation}
\end{subequations}
It is easy to see that the boundary conditions for $n$ imply $\int_\Omega n\,\d{x} = \int_\Omega n_0\,\d{x}$. Therefore, it is sufficient to ensure that $\int_\Omega n_0\,\d{x}=1$ for $n$ to be kept normalized.

Such kind of systems occur typically in the theory of transport equations for semiconductors as a macroscopic limit (quantum drift-diffusion equations) of its microscopic counterpart (Wigner-Boltzman-Poisson system) (c.f.~\cite{ansgar, ansgar2} and references therein). In this setting, $n$ describes the electron density and therefore $\s<0$, since electrons are negatively charged. Instead of (\ref{eq::system}c), the {\em electrostatic potential} $\Phi$ satisfies
\[
 -\Delta\Phi = n - C\;\;\mbox{in}\;\Omega,
\]
where $C$ is the doping profile of the semiconductor device. In the case of no interactions, $F= -\Delta\sqrt{n}/\sqrt{n}$, also called the Derrida-Lebowitz-Speer-Spohn (DLSS) equations \cite{derrida}, the authors in \cite{ugo} and \cite{matthes} successfully proved the existence of solutions and additionally showed the rate of convergence of solutions to the unique stationary solution. Existence results for system (\ref{eq::system}) with Neumann boundary conditions for general Lipschitz domains  were recently proven in \cite{xu} (see also \cite{chen}). In comparison to the proofs in \cite{xu}, we make use of an exponential transformation of the normalized density $n$, which we describe in Section~2. Note that the results in this paper hold also for the case $C\ne 0$, $C\in H^1(\Omega)\cap L^\infty(\Omega)$.

%either the classical drift-diffusion equations ($\s<0$) \cite{ansgar} or 
The case $\s>0$ on the other can be thought of as a macroscopic model for semi-classical quantum gravitating systems in flat space, where quantum effects of matter and classical approximation of the gravitational field is considered (c.f.~\cite{winter} and references therein). By passing to the limit $\e\to 0$, we formally recover a model for a system of self-gravitating particles \cite{biler}. This system is well known to have global solutions for $\s\in(-\infty,8\pi)$ and blowup phenomena for large mass $\s>8\pi$ in $d=2$. Another model of the limit system comes from statistical mechanics for vortex points as the mean field limit of the canonical Gibbs measure associated to a $N$-vortex system in a bounded two-dimensional domain \cite{caglioti1, caglioti2}. Here, the authors established the concentration phenomena for the weak limits of the Gibbs measures, when $N\to\infty$, to the solution of the limit system ($\e=0$) in the case $\s\in(-\infty,8\pi)$. Similar systems that manifest blowup phenomena arises in the theory of combustion \cite{bebernes} and chemotaxis equations \cite{suzuki}. In many cases, especially in higher dimensions, if a problem is presented with exponential nonlinearity working against diffusion, blowup phenomena occurs.  The modifications of the coupled matter-gravity problem by quantum mechanics are particularly interesting as they may result in a prevention of gravitational colapse, otherwise inevitable due to the singularity theorems. We note that the techniques used in \cite{xu} for $\s<0$ may not be directly applicable in this case.

The main objective of this paper is to show that the presence of the Bohm potential ($\e>0$) leads to a regularization of the limit problem ($\e=0$), and therefore does indeed prevent gravitational colapse ($\s>0$) as postulated. This would imply that the presence of the Bohm potential prevents blowup phenomena. More precisely, it is shown that (\ref{eq::system}) possesses at least one global weak solution for any $\s\in\rr$. A result on stationary solutions to (\ref{eq::system}) for arbitrary $\s\in\rr$ can be found in \cite{tse}. There it is also shown that stationary solutions with $\e>0$, $d=2$ and $\s\in(-\infty,8\pi)$, converge in the weak sense to stationary solutions of the limit problem ($\e=0$).

For notational convenience, we set
\[
 V := \left\{u\in H^2(\Omega)\;|\; \partial_\nu u=0\;\;\mbox{on}\;\Gamma\right\},\quad V_0:= H^2(\Omega)\cap H_0^1(\Omega),
\]
and denote the space of non-negative integrable functions with finite entropy $\mathcal{E}$ by
\[
 \mathcal{P} := \left\{ u\in L^1(\Omega)\;|\; u\ge 0,\; \int_\Omega u\,\d{x}=1,\; \mathcal{E}(u)<+\infty \right\},
\]
where $\mathcal{E}$ is simply the (negative) \emph{physical entropy} given by
\[
 \mathcal{E}(u) := \int_\Omega \big(u(\log u - 1) + 1 \big)\,\d{x} \ge 0.
\]
The main result of this paper is the following.

\begin{theorem}\label{thm::existence}
 For any finite $T>0$ and $n_0\in\mathcal{P}$, there exists at least one weak solution $(F,n,\Phi)$ to system (\ref{eq::system}), with 
 \begin{align*}
  n\in W^{1,1}(0,T;V^*),\quad \sqrt{n}\in L^2(0,T;V),\hspace*{1.6em}\\
  F\sqrt{n}\in L^2(0,T;L^2(\Omega))\quad\mbox{and}\quad\Phi\in L^1(0,T;V_0),
 \end{align*}
 and additionally $n\ge 0$ a.e.~in $(0,T)\times\Omega$, satisfying
 \begin{align*}
  \langle n_t,\test_1\rangle_{V^*,V} &= \int_\Omega F\sqrt{n}\big(\sqrt{n}\Delta\test_1 + 2\nabla \sqrt{n}\.\nabla \test_1\big)\d{x} \\
  \int_\Omega F\sqrt{n}\,\test_2\,\d{x} &= \int_\Omega \big(-\e^2\Delta\sqrt{n} + \sqrt{n}\log n - \s\sqrt{n}\,\Phi\big)\d{x} \\
  \int_\Omega \nabla \Phi\.\nabla \test_3\,\d{x} &= \int_\Omega n\,\test_3\,\d{x}
 \end{align*}
  for a.e.~$t\in(0,T)$ and all $\test=(\test_1,\test_2,\test_3)\in V\times L^2(\Omega)\times H^1_0(\Omega)$.
\end{theorem}

The proof of Theorem~\ref{thm::existence} relies on the fact that the physical entropy $\mathcal{E}$ provides a controlled growth estimate for (\ref{eq::4th_order}). Indeed, by formally multiplying (\ref{eq::4th_order}a) with $\log(n)$, integrating over $\Omega$ and integrating by parts, we obtain
\[
 \frac{\d}{\d{t}}\mathcal{E}(n) + c_0\int_\Omega \left|\nabla^2\sqrt{n}\right|^2\,\d{x} \le c_1,
\]
for some constants $c_0,c_1>0$, independent of $n$, which is essentially due to Propositions~\ref{prop::n2inequality} and \ref{prop::dummy}. By using well-known interpolation inequalities on derivatives \cite{adams}, one obtains constants $\delta>0$ and $c(\delta)>0$ such that
\[
 \|\nabla \sqrt{n}\|_2^2 \le \delta \|\nabla^2 \sqrt{n}\|_2^2 + c(\delta)\|\sqrt{n}\|_2^2,
\]
along with the logarithmic-Sobolev inequality \cite{stroock}
\[
 \int_\Omega n\log\Big(\frac{n}{\|\sqrt{n}\|_2^2}\Big)\,\d{x} \le c_L\|\nabla \sqrt{n}\|_2^2,
\]
where $c_L>0$ only depends on $\Omega$ and $d$, and the fact that $\|n\|_1=1$, we further obtain
\[
 \frac{\d}{\d{t}}\mathcal{E}(n) + c_2\mathcal{E}(n) \le c_3,
\]
for some constants $c_2,c_3>0$. An application of Grownwall's lemma on this estimate provides the global boundedness in time $t$ of $\mathcal{E}(n)$, and consequently shows (formally) the absence of a blowup phenomena in the space $\mathcal{P}$.

The strategy for a rigorous prove of this statement involves first showing existence of solutions for a time-discrete problem with time step $\tau>0$ with the help of Leray--Schauder's fixed point theorem. Section~\ref{sec::prelim} is devoted to recall results on elliptic equations required for the time-discrete problem. In Section~\ref{sec::discrete} we establish an important uniform entropy estimate (c.f.~(\ref{lem::discrete::eq::estimates})), which leads to the solvability of the time-discrete problem. Consequently, by establishing uniform bounds on the sequence of solutions $\{n^{(\tau)}\}$ with respect to $\tau$ in Section~\ref{sec::uniform}, we may then extract a subsequence, which converges to a solution of (\ref{eq::system}) when passing to the limit $\tau\to 0$. This final step is shown in Section~\ref{sec::limit}. 

\section{Preliminary results}\label{sec::prelim}

In this section we recall several standard results regarding the unique solvability and regularity of solutions for elliptic equations. We begin by recalling a well known interpolation theorem for Sobolev spaces, namely the Gagliardo--Nirenberg--Sobolev inequalities \cite{zeidler}.

\begin{proposition}\label{prop::gagliardo} 
 Let $\Omega\subset\rr^d$ be a bounded Lipschitz domain, $m\in\nn$ and $1\le p$, $q$, $r\le\infty$. Then there exists a constant $c>0$, such that
 \[
  \|D^{\alpha}u\|_r \le c\,\|u\|_{m,p}^\theta\|u\|_q^{1-\theta}\quad\mbox{for any}\;\; u\in W^{m,p}(\Omega)\cap L^q(\Omega),
 \]
 provided that $0\le |\alpha|\le m-1$, $\theta=|\alpha|/m$ or $|\alpha|-d/r=\theta(m-d/p)-(1-\theta)d/q$. If $m-|\alpha|-d/p\ne \nn_0$, then the values $|\alpha|/m \le \theta \le 1$ are allowed.
\end{proposition}

We also recall a regularity result for linear elliptic problems on convex, bounded domains due to \cite{grisvard}.
\begin{proposition}\label{prop::H^2}
 Let $\Omega\subset\rr^d$ be a convex, bounded domain and $f\in L^2(\Omega)$. Then the homogeneous Dirichlet problem
 \[
  -\Delta u = f\;\;\mbox{in}\;\Omega,\quad u=0\;\;\mbox{on}\;\Gamma,
 \]
 possesses a unique weak solution $u\in V_0$, which satisfies the estimate
 \[
  \|u\|_{2,2} \le c\|\Delta u\|_2,
 \]
 for some constant $c>0$, depending only on the diameter of $\Omega$.
\end{proposition}

The next result we want to recall is a regularity result for linear elliptic equations with homogeneous Neumann boundary conditions and source terms from the Orlicz space $L\log L(\Omega)$ (c.f.~\cite{brezis, stamm}). For the sake of completeness, we include a simple proof for this result in Appendix~\ref{appendix::LlogL}.

\begin{proposition}\label{prop::LlogL}
 Let $\Omega\subset\rr^d$ be a bounded Lipschitz domain, $f\in L\log L(\Omega)$ and $a\in L^\infty(\Omega)$ with $a\ge \alpha>0$ a.e.~in $\Omega$. Then the elliptic equation
 \begin{equation}\label{eq::LlogL}
  -\div\big( a\nabla u \big) = f\;\;\mbox{in}\;\Omega,\quad \partial_\nu u=0\;\;\mbox{on}\;\Gamma,
 \end{equation}
 has a unique weak solution $u\in W_\beta$, with $(1 +|u|)^{1/2}\in H^1(\Omega)$, where
 \[
  W_\beta := \left\lbrace u\in W^{1,\frac{d}{d-1}}(\Omega)\;|\; \frac{1}{|\Omega|}\int_\Omega u\,\d{x}=\beta \right\rbrace\quad\mbox{with}\quad \|u\|_{W_\beta} := \|\nabla u\|_{\frac{d}{d-1}},
 \]
 and $\beta\in\rr$ is some given constant.
\end{proposition}

Another result we will need is an existence and regularity result regarding nonlinear elliptic equations with natural gradient growth \cite{tse} (c.f.~\cite{dall, ferone}).

\begin{proposition}\label{prop::natural_gradient}
 Let $\Omega\subset\rr^d$ be a bounded Lipschitz domain and $f\in L^p(\Omega)$ with $p>d/2$. Then the elliptic equation
 \[
  -\frac{\e^2}{2}\left(\Delta y + \frac{1}{2}|\nabla y|^2\right) + y = f\;\;\mbox{in}\;\;\Omega,\quad 
  \partial_\nu y=0\;\;\mbox{on}\;\;\Gamma,
 \]
 has a weak solution $y\in H^1(\Omega)\cap L^\infty(\Omega)$. Furthermore, $e^{y/2}\in H^1(\Omega)\cap L^\infty(\Omega)$.
\end{proposition}

\section{Existence of a time-discrete solution}\label{sec::discrete}

Let $\tau>0$ an arbitrary but fixed time step and $w\in\mathcal{P}$ be a given function. The task at hand is to find a weak solution $(F,\sqrt{n},\Phi)\in W_\beta\times V\times V_0$, for some $\beta\in\rr$, to the semi-discrete system
\begin{subequations}\label{eq::discrete}
 \begin{align}
  -\div\big(n\nabla F\big) &= (w-n)/\tau \;\;\mbox{in}\;\Omega,\quad\hspace*{0.05em} n\partial_\nu F = 0 \;\;\mbox{on}\;\Gamma, \\
  -\e^2\frac{\Delta\sqrt{n}}{\sqrt{n}} + \log{n} &= \s\Phi + F \;\;\hspace*{1.06em}\mbox{in}\;\Omega,\quad
 \partial_\nu\sqrt{n} = 0 \;\;\mbox{on}\;\Gamma, \\
  -\Delta \Phi &= n \;\;\hspace*{3.8em}\mbox{in}\;\Omega,\quad 
  \hspace*{1.7em}\Phi = 0 \;\;\mbox{on}\;\Gamma,
 \end{align}
\end{subequations}
where $F$, $n$ and $\Phi$ satisfy (\ref{eq::system}a), (\ref{eq::system}b) and (\ref{eq::system}c) respectively.

We mention some of the problems encountered in solving this problem. Observe that, in order to solve (\ref{eq::discrete}a) for $F$, we have to make sure that $n\in L^\infty(\Omega)$ is uniformly bounded below away from zero. Otherwise, we will have to rely on degenerate elliptic estimates, which will require some kind of regularity on both $n$ and $n^{-1}$. 

To circumvent both these problems simultaneously, we make use of a transformation for $n$ initially introduced in \cite{ansgar2}, given by $n=e^y/\|e^y\|_1$. Clearly, if $y\in L^\infty(\Omega)$, then $n$ satisfies the assumptions in Proposition~\ref{prop::LlogL}. Moreover, the normality of $n$ is satisfied trivially. Inserting this into (\ref{eq::discrete}b), leads to a nonlinear elliptic equation with natural gradient growth,
\begin{equation}\label{eq::elliptic}
 -\frac{\e^2}{2}\left(\Delta y + \frac{1}{2}|\nabla y|^2\right) + y = \s\Phi + F + \log{\|e^y\|_1} \;\;\mbox{in}\;\Omega, \quad\partial_\nu y = 0 \;\;\mbox{on}\;\Gamma.
\end{equation}
We will then use (\ref{eq::elliptic}) to construct an auxiliary problem, similar to (\ref{eq::discrete}), and apply the Leray--Schauder theorem on this problem. Under the appropriate regularity, we may then recover solutions of the original problem (\ref{eq::discrete}).

\begin{lemma}\label{lem::discrete}
 Let $d\le 3$ and $w\in \mathcal{P}$. Then there exists at least one weak solution 
 \[
  (F,\sqrt{n},\Phi)\in W_\beta\times V\times V_0,
 \]
 to (\ref{eq::discrete}) with $n=e^y/\|e^y\|_1\in\mathcal{P}$ for some $y\in V$ and $\beta=-\log \|e^y\|_1$.
\end{lemma}
\begin{proof}
 Let $w\in \mathcal{P}$. For arbitrarily given $v\in \mathcal{C}(\overline{\Omega})$ and $\lambda\in[0,1]$, we consider the auxiliary problem to find $(F,y,\Phi)\in W_\beta\times V\times V_0$, with $\beta=-\lambda\log\|e^v\|_1$:
 \begin{subequations}\label{eq::auxiliary}
 \begin{align}
  -\div\big((e^v/\|e^v\|_1)\nabla F\big) &= \lambda\big(w-(e^v/\|e^v\|_1)\big)/\tau\;\;\mbox{in}\;\Omega,\quad
  \partial_\nu F = 0 \;\;\mbox{on}\;\Gamma, \\
  -\frac{\e^2}{2}\left(\Delta y + \frac{1}{2}|\nabla y|^2\right) + y &= \s\Phi + F + \lambda\log{\|e^v\|_1} \;\;\mbox{in}\;\Omega, \quad\hspace*{0.25em}\partial_\nu y = 0 \;\;\mbox{on}\;\Gamma, \\
  -\Delta \Phi &= \lambda \big(e^{v}/\|e^v\|_1\big) \;\;\hspace*{3.8em}\mbox{in}\;\Omega,\quad 
  \hspace*{1.15em}\Phi = 0 \;\;\mbox{on}\;\Gamma,
 \end{align}
 \end{subequations}
As mentioned above, we used the variable transformation $n=e^y/\|e^y\|_1$. Note also that for the case $\lambda=1$, system (\ref{eq::auxiliary}) is equivalent to the initial system (\ref{eq::system}).

We begin by showing that this system yields a weak solution $y\in V$ for any $v\in \mathcal{C}(\overline{\Omega})$ and $\lambda\in[0,1]$. Indeed, since $e^v/\|e^v\|_1\in \mathcal{C}(\overline{\Omega})$ and is uniformly bounded from below by a positive constant, we obtain a unique solution $F\in W_\beta$ for (\ref{eq::auxiliary}a) due to Proposition~\ref{prop::LlogL}. From Proposition~\ref{prop::H^2}, we obtain a unique solution $\Phi\in V_0$ for (\ref{eq::auxiliary}c). Now, by inserting $(F,\Phi)$ into (\ref{eq::auxiliary}b), we obtain a solution $y\in H^1(\Omega)\cap L^\infty(\Omega)$, due to Proposition~\ref{prop::natural_gradient}. Note also that $n=e^y/\|e^y\|_1\in L^\infty(\Omega)$ with $n\ge e^{\inf y}/\|e^y\|_1=:\underline{n}> 0$ a.e.~in $\Omega$ and $\int_\Omega n\,\d{x} = 1$. Moreover, by rewriting (\ref{eq::auxiliary}b) in terms of $\rho:=\sqrt{n}$, we have %, and $\int_\Omega n\,\d{x} = 1$
\[
 -\e^2\Delta\rho = \rho\left(\s\Phi + F - \log\rho\right)\in L^2(\Omega),
\]
which implies $\Delta\rho\in L^2(\Omega)$, since $\rho\in L^\infty(\Omega)$. Due to the convexity of $\Omega$ and the homogeneous Neumann boundary condition for $\rho$, we have the second order inequality $\|\nabla^2 \rho\|_2 \le \|\Delta \rho\|_2$, where $\nabla^2\rho$ denotes the Hessian of $\rho$ \cite{grisvard}. Therefore, $\rho\in V$ and consequently $y\in V$, since
\begin{equation}\label{lem::discrete::eq::regularity}
 \partial_{ij} y = 2\partial_i\Big(\rho^{-1}\partial_j\rho\Big) = 2\big(\rho\,\partial_{ij}\rho - \partial_i\rho\,\partial_j\rho\,\big)/\rho^2 \in L^2(\Omega),
 %\quad \partial_\nu y=0\;\;\mbox{a.e.~on}\;\Gamma.
\end{equation}
Due to the boundary condition for $y$, we further have $\partial_\nu \rho =0$ a.e.~on $\Gamma$, i.e., $\rho\in V$.

Now consider the operator $\mbox{H}\colon \mathcal{C}(\overline{\Omega})\times[0,1]\to \mathcal{C}(\overline{\Omega});\; (v,\lambda)\mapsto y$. This operator is continuous and compact due to the continuity of the solution operators corresponding to (\ref{eq::auxiliary}a)--(\ref{eq::auxiliary}c) respectively and the compact embedding $H^2(\Omega)\hookrightarrow\mathcal{C}(\overline{\Omega})$. It is also easy to see that $\mbox{H}(v,0)=0$ for all $v\in \mathcal{C}(\overline{\Omega})$. We see this by simply testing the variational formulation of (\ref{eq::auxiliary}b) with $\test=\mbox{sign}(y)(e^{|y|}-1)$, which yields
\[
 \e^2\int_\Omega |\nabla e^{|y|/2}|^2\,\d{x} + \frac{\e^2}{4}\int_\Omega |\nabla y|^2\,\d{x} + \int_\Omega |y|^2\,\d{x} \le 0,
\]
where we used the fact that $s(e^s-1) \ge s^2$ for $s\ge 0$. Therefore, $y=0$ a.e.~in $\Omega$ and consequently $y=0$ in $\overline{\Omega}$, since $y$ is continuous.

Let $(y,\lambda)\in \mathcal{C}(\overline{\Omega})\times[0,1]$ such that $\mbox{H}(y,\lambda)=y$. We now show that $y$ is uniformly bounded in $\mathcal{C}(\overline{\Omega})$ w.r.t.~$\lambda$ by some constant $M>0$. Observe that for $\sqrt{n}\in V$, the auxiliary system (\ref{eq::auxiliary}) is equivalent to the equations
\begin{subequations}\label{lem::discrete::eq::uniform}
\begin{align}
 \frac{\lambda}{\tau}(n-w) &= \div\left(n\nabla\left(-\e^2\frac{\Delta\sqrt{n}}{\sqrt{n}} + \log{n} - \s\Phi\right)\right) \\
 -\Delta\Phi &= \lambda n,
\end{align}
\end{subequations}
where we also used the fact that $\log{\|e^y\|_1}$ is constant. Since $\phi(s) = s(\log{s}-1) + 1$ is convex, $\phi(s_1)-\phi(s_2)\le \phi'(s_1)(s_1-s_2)$ for all $s_1,s_2\ge 0$. Therefore,
\begin{align*}
 \frac{\lambda}{\tau}(\mathcal{E}(n)-\mathcal{E}(w)) &= \frac{\lambda}{\tau}\int_\Omega (\phi(n)-\phi(w))\,\d{x} \le \frac{\lambda}{\tau}\int_\Omega (n-w)\log{n}\,\d{x} \\
 &\hspace*{-6em}= -\left\langle n\nabla\left(-\e^2\frac{\Delta\sqrt{n}}{\sqrt{n}} + \log{n}-\s\Phi\right), \nabla\log{n} \right\rangle \\
 &\hspace*{-6em}= \e^2\left\langle n\nabla\left(\frac{\Delta\sqrt{n}}{\sqrt{n}}\right),\nabla\log{n}\right\rangle - \int_\Omega n|\nabla\log{n}|^2\,\d{x} + \s\int_\Omega \nabla\Phi\.\nabla n\,\d{x} \\
 &\hspace*{-6em}= \e^2 J_1 - J_2 + \s J_3.
\end{align*}
For the following computations we set $\rho=\sqrt{n}$. For $J_1$, we have the following
\begin{align*}
 J_1 &= 2 \left\langle \nabla\big(\rho\,\Delta\rho\,\big) - 2\Delta\rho\,\nabla\rho, \frac{\nabla\rho}{\rho}\right\rangle \\
 &= -2 \int_\Omega |\Delta\rho|^2 + \Delta\rho\frac{|\nabla\rho|^2}{\rho}\,\d{x}
 \le -2\int_\Omega \rho^2\left[\left|\frac{\nabla^2\rho}{\rho}\right|^2 + \frac{\Delta\rho}{\rho}\left|\frac{\nabla\rho}{\rho}\right|^2\right]\d{x},
\end{align*}
where, in the last inequality, we used the fact that $\|\nabla^2\rho\|_2\le \|\Delta\rho\|_2$, due to the results in \cite{grisvard}.
As for $J_2$, we have the \emph{Fisher information} 
\[
 J_2 = \int_\Omega n|\nabla\log{n}|^2\,\d{x} = 4\int_\Omega \rho^2\left|\frac{\nabla\rho}{\rho}\right|^2\d{x}.
\]
Since $\Phi$ satisfies (\ref{lem::discrete::eq::uniform}b) and the homogeneous boundary condition $\Phi=0$ on $\Gamma$, we have $\partial_\nu\Phi\le 0$ simply by the weak maximum principle. Therefore,
\[
 J_3 = -\int_\Omega (\Delta\Phi) n\,\d{x} + \int_{\Gamma}n\,\partial_\nu\Phi\,\d{s} = \lambda\int_\Omega n^2\,\d{x} + \int_{\Gamma}n\,\partial_\nu\Phi\,\d{s} \le \int_\Omega n^2\,\d{x}
\]
 for any $\lambda\in[0,1]$. Altogether, we have
 \[
  \frac{\lambda}{\tau}(\mathcal{E}(n)-\mathcal{E}(w)) \le -2\int_\Omega \rho^2\left[\e^2\left|\frac{\nabla^2\rho}{\rho}\right|^2 + \e^2\frac{\Delta\rho}{\rho}\left|\frac{\nabla\rho}{\rho}\right|^2 + 2\left|\frac{\nabla\rho}{\rho}\right|^2 - \frac{\s}{2}n\right]\d{x}.
 \]
 To bring the last term on the right-hand side into the same form as the rest of the terms, we apply the following result proved in Appendix~\ref{appendix::n2inequality}.
 \begin{proposition}\label{prop::n2inequality}
  Let $u\in H^2(\Omega)$ with $k\le u\le k^{-1}$ for some $k\in(0,1)$. Then for any $\delta_i>0$, $i=1,2$, there exists a constant $c>0$ such that
 \[
  \int_\Omega u^4\,\d{x} \le \delta_1\int_\Omega u^2\left|\nabla \log u\right|^4\,\d{x} + \delta_2\int_\Omega \left|\nabla u\right|^2\,\d{x} + c(\delta_1,\delta_2,\|u\|_2),
 \]
 where the constant $c$ depends also on the constants from Proposition~\ref{prop::gagliardo}.
 \end{proposition}
 
 Since $\rho\in V\cap L^\infty(\Omega)$ is essentially bounded from bellow and $\|\rho\|_2=1$, we have
 \[
  \int_\Omega n^2\,\d{x} \le \delta_1\int_\Omega \rho^2\left|\frac{\nabla\rho}{\rho}\right|^4\d{x} + \delta_2\int_\Omega \rho^2\left|\frac{\nabla\rho}{\rho}\right|^2\d{x} + c(\delta_1,\delta_2),
 \]
 for any $\delta_i>0$, $i=1,2$. By choosing $\delta_2=4/|\s|$ and rescaling $\delta_1 =2\e^2\delta/|\s|$ for some $\delta>0$, we further obtain
 \begin{equation}\label{lem::discrete::eq::first}
  \quad\frac{\lambda}{\tau}(\mathcal{E}(n)-\mathcal{E}(w)) + 2\e^2\int_\Omega \rho^2\left[\left|\frac{\nabla^2\rho}{\rho}\right|^2 + \frac{\Delta\rho}{\rho}\left|\frac{\nabla\rho}{\rho}\right|^2 - \delta\left|\frac{\nabla\rho}{\rho}\right|^4\right]\d{x} \le c(\delta).
 \end{equation}
 As a matter of fact, the second term on the left can be bounded from below by a multiple of $\|\nabla^2\rho\|_2^2$. More precisely, we have the following result proved in Appendix~\ref{appendix::dummy}.
\begin{proposition}\label{prop::dummy}
  Suppose $u\in V$ and $k\le u\le k^{-1}$ for some $k\in(0,1)$. Then for sufficiently small $\delta>0$, there exists a constant $\gamma>0$ such that
  \[
   \int_\Omega u^2\left[\left|\frac{\nabla^2 u}{u}\right|^2 + \frac{\Delta u}{u}\left|\frac{\nabla u}{u}\right|^2 - \delta\left|\frac{\nabla u}{u}\right|^4\right]\d{x} \ge \gamma \int_\Omega |\nabla^2 u|^2\,\d{x},
  \]
  where $\gamma=(1 + (d-1)c_0)/d$ with $c_0\in (0,3/(d+2))$.
\end{proposition}

Using this result on inequality (\ref{lem::discrete::eq::first}) gives us
 \begin{equation}\label{lem::discrete::eq::bound}
  \frac{\lambda}{\tau}(\mathcal{E}(n)-\mathcal{E}(w)) + 2\gamma\e^2\int_\Omega |\nabla^2\rho|^2\,\d{x} \le c(\delta)\quad\mbox{for any}\;\;\lambda\in[0,1],
 \end{equation}
 which leads to the uniform bound
 \begin{equation}\label{lem::discrete::eq::estimates}
  \|\nabla^2\rho\|_2^2 \leq \frac{1}{2\gamma\e^2\tau}\left(\mathcal{E}(w)+c(\delta)\tau\right),
 \end{equation}
 thus implying the boundedness of $\nabla^2\rho$ in $L^2(\Omega)$ independent of $\lambda\in[0,1]$, and therefore the uniform boundedness of $\rho\in V$. As a consequence, we obtain an upper bound for $y$, i.e. there exists some constant $M_u>0$, such that $y< M_u$ a.e.~in $\Omega$.
 
 To show that $y$ is uniformly bounded from below, we have to show a uniform lower bound for $\rho$ away from zero. This is a result of the Harnack inequality \cite{serrin, trudinger}. Firstly, note that $\log\rho -\s\Phi -F\in L^p(\Omega)$, $p>d/2$. Then by Harnack's inequality, a weak solution $\rho\in H^1_{loc}(\Omega)$ with $0\le \rho\le M_u$ in $Q$ of the equation
 \[
  -\e^2\Delta\rho + \mu\rho = 0\;\;\mbox{in}\;\Omega,
 \]
 with $\mu=\mu(\lambda)\in L^p(\Omega)$, $p>d/2$, in a cube $Q=Q(3r)\subset\Omega$ satisfies
 \[
  \max_{Q(r)}\rho(x) \le c_h\min_{Q(r)} \rho(x),
 \]
 for some constant $c_h=c_h(\lambda)>0$ independent of $\rho$. 
 
 Now it is an easy exercise to verify that having $\rho(x)=0$ for some $x\in \Omega$ would lead to $\rho\equiv 0$ in $\Omega$, which clearly contradicts $\|\rho\|_2=1$. Therefore, $\rho\ge \underline{\rho}>0$ a.e.~in $\Omega$ uniformly in $\lambda$ and there is some constant $M_l>0$ such that $y> -M_l$ a.e.~in $\Omega$. Choosing $M=\max\{M_l,M_u\}$ gives the estimate $\|y\|_\infty < M$. Furthermore, we see from (\ref{lem::discrete::eq::regularity}) that $\nabla^2 y\in L^2(\Omega)$ is uniformly bounded. This implies that $y$ is uniformly bounded in $H^2(\Omega)$, and consequently in $\mathcal{C}(\overline{\Omega})$ with the same constant $M$, due to the Sobolev embedding $H^2(\Omega)\hookrightarrow \mathcal{C}(\overline{\Omega})$. We finally conclude the proof by applying the Leray--Schauder fixed point theorem \cite{gilbarg}.
\end{proof}

The value $\beta\in\rr$ may be thought of as a Lagrange multiplier for the constraint on the density $\int_\Omega n\,\d{x}=1$. Indeed, the solution $\sqrt{n}\in \mathcal{P}$ of (\ref{eq::discrete}b) may be characterized as the unique minimizer of the functional
\[
 \mathcal{F}(n):= \e^2\int_\Omega |\nabla\sqrt{n}|^2\,\d{x} + \mathcal{E}(n) - \int_\Omega n\left(\s\Phi + F\right)\,\d{x},
\]
on the space $\{n\in\mathcal{P}\;|\; \sqrt{n}\in V\}$ for given $(F,\Phi)$.

\section{Uniform estimates in $\tau$}\label{sec::uniform}

Let $T>0$ be a fixed arbitrary terminal time. For every $\tau>0$ we define the step function $n^{(\tau)}\colon[0,T)\to L^1(\Omega)$ recursively as follows. Let $n(0)=n_0$ and for given $k\in\nn$, let $(F_k,\sqrt{n_k},\Phi_k)\in W_\beta\times V\times V_0$ be a solution of (\ref{eq::discrete}) with $w=n_{k-1}$ and $n_k$ positive. Now define 
\[ 
 F^{(\tau)}(t):=F_k,\;\; n^{(\tau)}(t):=n_k,\;\; \Phi^{(\tau)}(t):=\Phi_k\quad\mbox{for}\quad (k-1)\tau< t\le k\tau.
\]
Then $(F^{(\tau)},n^{(\tau)},\Phi^{(\tau)})$ satisfies
\begin{subequations}\label{eq::evolution}
 \begin{align}
  (n^{(\tau)} - \zeta_\tau n^{(\tau)})/\tau &= \div\big(n^{(\tau)}\nabla F^{(\tau)}\big)\;\;\hspace*{6.6em}\mbox{in}\;\Omega,\\
  F^{(\tau)} &= -\e^2\frac{\Delta\sqrt{n^{(\tau)}}}{\sqrt{n^{(\tau)}}} + \log{n^{(\tau)}} - \s\Phi^{(\tau)} \;\;\mbox{in}\;\Omega, \\
  -\Delta \Phi^{(\tau)} &= n^{(\tau)} \;\;\hspace*{11.6em}\mbox{in}\;\Omega,
 \end{align}
\end{subequations}
together with their respective boundary conditions. Here, $\zeta_\tau$ denotes the shift operator $(\zeta_\tau n^{(\tau)})(t)=n^{(\tau)}(t-\tau)$ for $t\in[\tau,T)$. As a consequence of Lemma~\ref{lem::discrete}, we obtain the following uniform estimate for the sequence of step functions $\{n^{(\tau)}\}$.

\begin{lemma}\label{lem::uniform}
 There exists a $\delta_0>0$ such that for any $\delta\in[0,\delta_0]$, the sequence of step functions $\{n^{(\tau)}\}$ satisfies
 \[
  \tau^{-1}\|n^{(\tau)} - \zeta_\tau n^{(\tau)}\|_{L^{1+\delta}(\tau,T;V^*)} + \|n^{(\tau)}\|_{L^{1+\delta}(0,T;H^2(\Omega))} \le c
 \]
 for some constant $c>0$ independent of $\tau$.
\end{lemma}
\begin{proof}
 For notational convenience, we set as before $\rho^{(\tau)}=\sqrt{n^{(\tau)}}$ and use $c>0$ as a generic constant independent of $\tau$. We begin by establishing a uniform bound in $\tau$ for the translations. To do so, we multiply (\ref{eq::evolution}a) with an arbitrary function $\varphi\in V$, integrate over $\Omega$ and integrate by parts to obtain
 \begin{align}\label{lem::uniform::eq::weak}
  \frac{1}{\tau}\int_\Omega \big(n^{(\tau)} - \zeta_\tau n^{(\tau)}\big)\varphi\,\d{x} &= \int_\Omega F^{(\tau)}\big(\nabla n^{(\tau)}\.\nabla\varphi + n^{(\tau)}\Delta\varphi\big)\,\d{x} \\
  &\le c\left(\|F^{(\tau)}\nabla n^{(\tau)}\|_\frac{4}{3} + \|F^{(\tau)} n^{(\tau)}\|_2\right)\|\varphi\|_V,\nonumber
 \end{align}
 where we used the embedding $H^2(\Omega)\hookrightarrow W^{1,4}(\Omega)$. This implies the estimate
 \[
  \frac{1}{\tau}\|n^{(\tau)} - \zeta_\tau n^{(\tau)}\|_{L^{1+\delta}(\tau,T;V^*)}^{1+\delta} \le c_1\int_0^T \|F^{(\tau)}\nabla n^{(\tau)}\|_\frac{4}{3}^{1+\delta} + \|F^{(\tau)}(t) n^{(\tau)}(t)\|_2^{1+\delta}\,\d{t}.
 \]
 Since there are constants $c_i>0$, $i=1,\ldots,4$, such that
 \[
  \|F\nabla n\|_\frac{4}{3}^{1+\delta}  \le c_1\|F\rho\|_2^2 + c_2\|\nabla\rho\|_4^{2\eta(\delta)}\quad\mbox{and}\quad \|F n\|_2 \le c_3\|F\rho\|_2^2 + c_4\|\rho\|_\infty^{2\eta(\delta)},
 \]
 with $\eta(\delta)=(1+\delta)/(1-\delta)>1$, we further obtain
 \begin{equation}\label{lem::uniform::eq::shift}
  \frac{1}{\tau}\|n^{(\tau)} - \zeta_\tau n^{(\tau)}\|_{L^{1+\delta}(\tau,T;V^*)}^{1+\delta} \le c\int_0^T \!\!\|\rho^{(\tau)}\|_\infty^{2\eta(\delta)} + \|\nabla\rho^{(\tau)}\|_4^{2\eta(\delta)} + \|F^{(\tau)}\rho^{(\tau)}\|_2^2\,\d{t}.
 \end{equation}
 Therefore, the required estimate follows from the uniform boundedness of the right hand side of the above equation in $\tau$.

 We now make several observations from the results obtained in Section~\ref{sec::discrete}. As a direct consequence of the conservation of mass, we obtain the uniform bound
 \[
  \|\rho^{(\tau)}\|_{L^\infty(0,T;L^2(\Omega))} = \|n^{(\tau)}\|_{L^\infty(0,T; L^1(\Omega))} = \int_\Omega n_0\,\d{x} = 1.
 \]
 Furthermore, we deduce from (\ref{lem::discrete::eq::bound}) the uniform bound
 \begin{equation}\label{lem::uniform::eq::0}
  \|\nabla^2\rho^{(\tau)}\|_{L^2(0,T;L^2(\Omega))} \le c_0,
 \end{equation}
 with a constant $c_0>0$ depending only on $d$, $\Omega$, $T$, and $n_0$, and consequently, the uniform boundedness of the sequence $\{\rho^{(\tau)}\}\subset L^2(0,T;V)$. In order to establish uniform boundedness of the first two terms in (\ref{lem::uniform::eq::shift}), we make use of Proposition~\ref{prop::gagliardo}. For the first term, we have
 \begin{equation}\label{lem::uniform::eq::1}
  \int_0^T \|\rho\|_\infty^{2\eta(\delta)}\,\d{t} \le c\,\|\rho\|_{L^\infty(0,T;L^2(\Omega))}^{2(1-\theta)\eta(\delta)}\int_0^T\|\rho\|_{2,2}^{2\theta\eta(\delta)}\,\d{t},
 \end{equation}
 with $\theta=d/4\in(0,1)$ for $d\le 3$, which leads to uniform boundedness if we choose $\eta(\delta)\le 4/d$. This is equivalent to the requirement $\delta\le (4-d)/(4+d)=:\delta_1$. Similarly, we have for the second term
 \begin{equation}\label{lem::uniform::eq::2}
  \int_0^T \|\nabla\rho\|_4^{2\eta(\delta)}\,\d{t} \le c\,\|\rho\|_{L^\infty(0,T;L^2(\Omega))}^{2(1-\theta)\eta(\delta)}\int_0^T\|\rho\|_{2,2}^{2\theta\eta(\delta)}\,\d{t}, 
 \end{equation}
 with $\theta = (4+d)/8\in(0,1)$ for $d\le 3$. Here, we may choose $\eta(\delta)\le 8/(4+d)$, which is equivalent to choosing $\delta\le (4-d)/(12+d)=:\delta_2$. Therefore, we may choose $\delta_0=\min\{\delta_1,\delta_2\}$. The uniform boundedness of the last term in (\ref{lem::uniform::eq::shift}) may be seen as follows. From (\ref{eq::evolution}b), we obtain the following estimate
 \[
  \|F\rho\|_2 \le \e^2\|\nabla^2\rho\|_2 + 2\|\rho\log \rho\|_2 + |\s|\|\rho\Phi\|_2 
  \le \e^2\|\rho\|_{2,2} + \big(2 + c\,|\s|\big)\|\rho\|_4^2,
 \]
 where we used the fact that $s\log(s)\le s^2$ for $s>0$, the a priori estimate provided by Proposition~\ref{prop::H^2}, and the Sobolev embedding $H^2(\Omega)\hookrightarrow L^\infty(\Omega)$. Now, from Proposition~\ref{prop::gagliardo}, we deduce the existence of constants $c_5,c_6>0$, such that
 \begin{equation}\label{lem::uniform::eq::3}
  \int_0^T \|F\rho\|_2^{2}\,\d{t} \le c_5\int_0^T \|\rho\|_{2,2}^2\,\d{t} + c_6\|\rho\|_{L^\infty(0,T;L^2(\Omega))}^{4(1-\theta)}\int_0^T \|\rho\|_{2,2}^{4\theta}\,\d{t},
 \end{equation}
 with $\theta=d/8\in(0,1)$, which shows the uniform bound for $\{F^{(\tau)}\rho^{(\tau)}\}$. Therefore, by collecting the estimates obtained above, we conclude the first assertion.
 
 To show that $\{n^{(\tau)}\}$ is uniformly bounded in $L^{1+\delta}(0,T;H^2(\Omega))$, we simply use the fact that $\nabla^2 n = 2(|\nabla\rho|^2 + \rho\nabla^2\rho)$ and Proposition~\ref{prop::gagliardo} again, to obtain
 \[
  \int_0^T\|\nabla^2 n\|_2^{1+\delta}\,\d{t}  \le c\int_0^T \|\nabla\rho\|_4^{2(1+\delta)} + \|\rho\|_\infty^{2\eta(\delta)} + \|\nabla^2\rho\|_2^2 \,\d{t}.
 \]
 Since $(1+\delta)\le \eta(\delta)$, we have the uniform boundedness of the right hand side due to the estimates (\ref{lem::uniform::eq::0}), (\ref{lem::uniform::eq::1}) and (\ref{lem::uniform::eq::2}). Similarly, we can show the uniform bound for $\{n^{(\tau)}\}$ in $L^{1+\delta}(0,T;L^2(\Omega))$, which leads to the estimate asserted for $\{n^{(\tau)}\}$.
\end{proof}

\section{Passing to the limit $\tau\to 0$}\label{sec::limit}

We recall a nonlinear version of Aubin's lemma found in \cite{dreher}.

\begin{proposition}\label{prop::dreher}
Let $X,B,Y$ be Banach spaces such that the embedding $X\hookrightarrow B$ is compact and the embedding $B\hookrightarrow Y$ is continuous. Furthemore, let $1\le p<\infty$, $r=1$, and let $(u_\tau)$ be a sequence of functions, which are constant on each subinterval $(t_{k-1},t_k)$, satisfying
\[
 \tau^{-1}\|u_\tau - \zeta_\tau n^{(\tau)}\|_{L^r(\tau,T;Y)} + \|u_\tau\|_{L^p(0,T;X)}\le c_0\quad\mbox{for all}\;\;\tau>0,
\]
where $c_0>0$ is independent of $\tau$. Then $(u_\tau)$ is relatively compact in $L^p(0,T;B)$.
\end{proposition}

A simple consequence of Proposition~\ref{prop::dreher} is the following result.

\begin{lemma}\label{lem::limit}
 There exists a nonnegative function $n\in W^{1,1}(0,T;V^*)$ such that the following convergences hold true for some subsequence of $\{n^{(\tau)}\}$:
 \begin{align*}
  \big( n^{(\tau)} - \zeta_\tau n^{(\tau)} \big)/\tau \rightharpoonup \partial_t n\;\;\mbox{in}\; L^{1+\delta}(0,T;V^*),\quad
  n^{(\tau)} \rightharpoonup n\;\;\mbox{in}\; L^{1+\delta}(0,T;H^2(\Omega)), \\
  n^{(\tau)}\to n\quad\mbox{and}\quad\nabla n^{(\tau)}\to \nabla n\;\;\mbox{a.e.~in}\;(0,T)\times\Omega, \hspace*{6.5em}
 \end{align*}
 for $\tau\to 0$, where $\rightharpoonup$ denotes the weak convergences in their respective spaces.
\end{lemma}
\begin{proof}
 Since the Bochner spaces $L^{1+\delta}(0,T;V^*)$ and $L^{1+\delta}(0,T;H^2(\Omega))$, with $\delta>0$ are reflexive, they are weakly sequentially compact. Due to Lemma~\ref{lem::uniform}, we obtain a subsequence of $\{n^{(\tau)}\}$ (not relabeled) and some $n\in L^{1+\delta}(0,T;H^2(\Omega))$ such that the second convergence holds true. Furthermore, Proposition~\ref{prop::dreher} provides yet another subsequence of $\{n^{(\tau)}\}$ such that
 \[
  n^{(\tau)}\to n\;\;\mbox{in}\;L^{1+\delta}(0,T;W^{1,4}(\Omega)),
 \]
 due to the compact embedding $H^2(\Omega)\hookrightarrow W^{1,4}(\Omega)$. In particular, we may extract a subsequence such that $n^{(\tau)}\to n$ and $\nabla n^{(\tau)}\to \nabla n$ almost everywhere. Moreover, since $n^{(\tau)}\ge 0$ uniformly almost everywhere for all $\tau$, the limit function also satisfies $n\ge 0$ almost everywhere.

 In a similar way, we obtain some $\chi\in L^{1+\delta}(0,T;V^*)$ such that
 \[
  \big( n^{(\tau)} - \zeta_\tau n^{(\tau)} \big)/\tau \rightharpoonup \chi\;\;\mbox{in}\; L^{1+\delta}(0,T;V^*)
 \]
 for some subsequence. In order to identify this limit with the time derivative of $n$, we multiply $n$ with arbitrary functions $v\in V$ and $\varphi\in \mathcal{C}_0^{\infty}(0,T)$, integrate over time and space, and integrating by parts w.r.t.~time, we obtain
 \begin{align*}
  \int_0^T \langle n,v\rangle_V\varphi_t\,\d{t} &= \lim_{\tau\to 0} \int_0^{T-\tau} \langle n^{(\tau)},v\rangle_V\frac{\varphi(t)-\varphi(t+\tau)}{\tau}\,\d{t} \\
  &= \lim_{\tau\to 0} \int_\tau^T \langle\frac{n^{(\tau)}-\zeta_\tau n^{(\tau)}}{\tau},v\rangle_V\varphi\,\d{t} = \int_0^T \langle\chi(t),v\rangle_V\varphi(t)\,\d{t}.
 \end{align*}
 Since the set $\{\varphi\,v\;|\; v\in V,\,\varphi\in\mathcal{C}_0^\infty(0,T)\}$ is dense in $L^q(0,T;V)$, $q = (1+\delta)/\delta$, we have by definition of the generalized time derivative that $\partial_t n=\chi$.
\end{proof}

\begin{lemma}\label{lem::limit_rho}
 There exists a nonnegative function $\rho=\sqrt{n}\in L^2(0,T;V)$ such that the following convergences hold true for some subsequence of $\{\rho^{(\tau)}\}$:
 \begin{align*}
  \rho^{(\tau)} \to \rho\;\;\mbox{in}\;L^{2\eta(\delta)}(0,T;W^{1,4}(\Omega)),\hspace*{2.6em}\\
  \rho^{(\tau)}\to \rho\quad\mbox{and}\quad \nabla\rho^{(\tau)}\to \nabla\rho\;\;\mbox{a.e.~in}\;(0,T)\times\Omega,
 \end{align*}
 for $\tau\to 0$, with the same $\delta>0$ as in Lemma~\ref{lem::limit}.
\end{lemma}
\begin{proof}
 Due to the boundedness of $\{\rho^{(\tau)}\}$ in $L^2(0,T;V)$, we obtain a weakly convergent subsequence (not relabeled) and some $\rho$ such that $\rho^{(\tau)}\rightharpoonup \rho$ in $L^2(0,T;V)$. Since $n^{(\tau)}\to n$ almost everywhere, so does $\rho^{(\tau)}=\sqrt{n^{(\tau)}}\to\sqrt{n}$ almost everywhere, which allows us to identify $\rho=\sqrt{n}\ge 0$ almost everywhere. Furthermore, we have that $\nabla\rho^{(\tau)}\to \nabla\rho$ almost everywhere. Indeed, this follows from
 \[
  2\rho^{(\tau)}\nabla\rho^{(\tau)} = \nabla n^{(\tau)} \to \nabla n = 2\rho\nabla\rho\;\;\mbox{a.e.~in}\;(0,T)\times\Omega.
 \]
 Hence, due to the boundedness of the sequence $\{\rho^{(\tau)}\}\subset L^{2\eta(\delta)}(0,T;W^{1,4}(\Omega))$ given by (\ref{lem::uniform::eq::1}) and (\ref{lem::uniform::eq::2}), along with the almost everywhere convergence of the sequence, we may apply the Lebesgue's dominated convergence in order to obtain strong convergence in the asserted space.
\end{proof}

With the preceding results, we may now pass to the limit $\tau\to 0$ in (\ref{eq::evolution}).

 {\em Proof of Theorem~\ref{thm::existence}}.
 Let us first establish a weak convergence for $\{F^{(\tau)}\rho^{(\tau)}\}$ and $\{\Phi^{(\tau)}\}$. Due to estimate (\ref{lem::uniform::eq::3}) we obtain $\chi\in L^2(0,T;L^2(\Omega))$ such that
 \[
  F^{(\tau)}\rho^{(\tau)} \rightharpoonup \chi\;\;\mbox{in}\;L^2(0,T;L^2(\Omega)).
 \]
 for some subsequence (not relabeled). Due to the almost everywhere convergence for $\rho^{(\tau)}$ given in Lemma~\ref{lem::limit_rho}, we may define $F$ such that $F\rho = \chi$ almost everywhere, where we set $F(t,x)=0$ when $\rho(t,x)=0$. As for $\{\Phi^{(\tau)}\}$, we obtain from Proposition~\ref{prop::H^2} the a priori estimate $\|\Phi\|_{2,2}\le c\|n\|_2$, which directly gives us the uniform bound for $\{\Phi^{(\tau)}\}$. Therefore, there exists some subsequence and $\Phi$ such that
 \[
  \Phi^{(\tau)} \rightharpoonup \Phi\;\;\mbox{in}\; L^{1+\delta}(0,T;V_0).
 \]
 Now, by multiplying (\ref{eq::evolution}) with the appropriate test functions, integrating over time and space, integrating by parts, and passing to the limit $\tau\to 0$, we finally obtain
 \begin{align*}
   \langle \partial_t n,\varphi_1\rangle &= \langle (F\rho)2\nabla \rho,\nabla\varphi_1\rangle + \langle (F\rho)\rho,\Delta\varphi_1\rangle \quad\forall \varphi_1\in L^q(0,T;V),\\
   \langle F\rho,\varphi_2\rangle &= \langle -\e^2\Delta \rho + 2\rho\log \rho - \s\rho\Phi,\varphi_2\rangle \quad\hspace*{1.3em}\forall \varphi_2\in L^2(0,T;L^2(\Omega)),\\
   \langle \nabla\Phi,\nabla\varphi_3\rangle &= \langle n,\varphi_3\rangle \quad\hspace*{11.15em}\forall\varphi_3\in H_0^1(\Omega),
 \end{align*}
 with $q=(1+\delta)/\delta$, which completes the proof.\quad\endproof

\appendix

\section{Proof of Proposition~\ref{prop::LlogL}}\label{appendix::LlogL}
 Without loss of generality, $\beta=0$. Otherwise, we simply make the shift $u' = u-\beta \in W_0$ and proceed with the proof for $u'$. From Poincar\'e's inequality
 \[
  \|u - \frac{1}{|\Omega|} \int_\Omega u \,\d{x}\|_p \le c_p \|\nabla u\|_p\quad\mbox{for any}\;\; u\in W^{1,p},
 \]
 we deduce that the norms $\|\.\|_{W_0}$ and $\|\.\|_{1,\frac{d}{d-1}}$ are equivalent. 

 We know from standard elliptic theory that a unique weak solution $u\in W_0$ of (\ref{eq::LlogL}) exists when $f\in L^2(\Omega)$ due to the Lax--Milgram theorem, the Poincar\'e inequality, and the continuous embedding $H^1(\Omega)\hookrightarrow W_0$. Since $L^2(\Omega)\hookrightarrow L\log L(\Omega)$ is dense, it is sufficient to show, due to the BLT theorem, that the solution operator $S\colon f\mapsto u$ is bounded with respect to the norms $\|\.\|_{L\log L(\Omega)}$ and $\|\.\|_{W_0}$ respectively.

 For $\phi(s):=\mbox{sign}(s)\log(1 + |s|)$ we set
 \[
  \Phi(s) := (1+|s|)\big(\log(1+|s|)-1\big), \quad\mbox{and}\quad \theta(s) := 2(1+|s|)^{1/2}.
 \]
 We define $\Psi(r)$ as the conjugate convex function of $\Phi(s)$, i.e., 
 \[
  \Psi(r):=\sup_s \big( rs - \Phi(s) \big),
 \]
 where the supremum is attained if and only if $r=\phi(s)$ \cite{krasno}. We first observe that $|u|\in W^{1,\frac{d}{d-1}}(\Omega)$ for any $u\in W^{1,\frac{d}{d-1}}(\Omega)$ \cite{kinder}. Moreover, we have the estimate
 \[
  \|\nabla|u|\|_{\frac{d}{d-1}} \le \|1 + |u|\|_{\frac{d}{d-2}}^{1/2}\|\nabla\theta(u)\|_2 \le \frac{\delta}{2}\|u\|_{\frac{d}{d-2}} + \frac{1}{2\delta}\|\nabla\theta(u)\|_2^2 + \frac{\delta}{2}|\Omega|^{(d-2)/d},
 \]
 for any $\delta>0$. Clearly, this follows directly from H\"older's and Young's inequality. Due to the Sobolev embedding $W^{1,\frac{d}{d-1}}(\Omega)\hookrightarrow L^{\frac{d}{d-2}}(\Omega)$ as well as the norm equivalence $\|\.\|_{W_0}\sim\|\.\|_{1,\frac{d}{d-1}}$, we may choose an appropriate $\delta>0$ to further obtain
 \begin{equation}\label{eq::LlogL_bound}
  \|\nabla|u|\|_{\frac{d}{d-1}} \le \tilde{c}\|\nabla\theta(u)\|_2^2 + \frac{1}{2\tilde{c}}|\Omega|^{(d-2)/d}\quad\mbox{for any}\;\;u\in W_0,
 \end{equation}
 for some constant $\tilde{c}>0$, depending only on $d$ and $|\Omega|$. 
 
 Now, by testing equation (\ref{eq::LlogL}) with $\phi(u)$, we obtain
 \[ 
  \alpha \int_\Omega |\nabla \theta(u)|^2\,\d{x} \le \int_\Omega \Phi(kf)\,\d{x} + \int_\Omega \Psi(\frac{1}{k}\phi(u))\,\d{x},
 \]
 which holds for any $k>0$. Since $\Phi$ and $\Psi$ are convex with $\Phi(0)=\Psi(0)=0$, we have $\Phi(k s) \le k\Phi(s)$ and $\Psi(k^{-1}\phi(s))\le k^{-1}\Psi(\phi(s)) = k^{-1}(s\phi(s) - \Phi(s))$. Together with (\ref{eq::LlogL_bound}) we get
 \[
  \|u\|_{W_0} \le \frac{\tilde{c}k}{\alpha}\int_\Omega \Phi(f)\,\d{x} + \frac{1}{\alpha \tilde{c}k}\int_\Omega \big((1+|u|) - \log(1+|u|)\big)\,\d{x} + \frac{1}{2\tilde{c}}|\Omega|^{(d-2)/d}.
 \]
 Using the Sobolev embedding $W^{1,\frac{d}{d-1}}(\Omega)\hookrightarrow L^1(\Omega)$ and choosing $k>0$ appropriately finally leads to the estimate
 \[
  \|S f\|_{W_0} = \|u\|_{W_0} \le c\Big( \|f\|_{L\log L(\Omega)} + M(|\Omega|)\Big)\quad\mbox{for all}\;\;f\in L^2(\Omega),
 \]
 where $c>0$ is a constant depending only on $d$, $|\Omega|$ and $\alpha$, and $M\colon\rr_+\to\rr_+$ is a monotonically increasing function of $|\Omega|$. Consequently, we may extend the solution operator $S$ to $\hat{S}\colon L\log L(\Omega)\to W_0$, which establishes the existence of solutions for (\ref{eq::LlogL}). Uniqueness follows by standard arguments involving the superposition and maximum principle, and thereby concludes the proof. \quad\endproof

\section{Proof of Proposition~\ref{prop::n2inequality}}\label{appendix::n2inequality}
From Proposition~\ref{prop::gagliardo}, we obtain
 \[
  \|u\|_4 \le c_1\|u\|_{1,2}^{\frac{d}{4}}\|u\|_2^{\frac{4-d}{4}}.
 \]
 On the other hand, by setting $v=\sqrt{u}\in H^2(\Omega)$ with $\sqrt{k}\le v\le 1/\sqrt{k}$, we obtain
 \[
  \|u\|_4^{\frac{1}{2}}=\|v\|_8 \le c_2\|v\|_{1,4}^{\frac{d}{8}}\|v\|_4^{\frac{8-d}{8}} = c_2\|v\|_{1,4}^{\frac{d}{8}}\|u\|_2^{\frac{8-d}{16}},
 \]
 and therefore 
 \[
  \|u\|_4 \leq c_2^2\|v\|_{1,4}^{\frac{d}{4}}\|u\|_2^{\frac{8-d}{8}}.
 \]
 Now let $\alpha\in(0,1)$ for $d=2$, $\alpha\in(1/3,1)$ for $d=3$, and
 \begin{align*}
  \|u\|_4^{2(1-\alpha)} &\le c_1^{2(1-\alpha)}\|u\|_{1,2}^{\frac{(1-\alpha)d}{2}}\|u\|_2^{\frac{(1-\alpha)(4-d)}{2}},\\ 
  \|u\|_4^{2(1+\alpha)} &\le c_2^{4(1+\alpha)}\|v\|_{1,4}^{\frac{(1+\alpha)d}{2}}\|u\|_2^{\frac{(1+\alpha)(8-d)}{4}}.
 \end{align*}
 Putting them together and applying H\"older's and Young's inequalities yield,
 \begin{align*}
  \|u\|_4^4 &\le c_1^{2(1-\alpha)}c_2^{4(1+\alpha)}\|u\|_2^{\frac{16-(3-\alpha)d}{4}}\|v\|_{1,4}^{\frac{(1+\alpha)d}{2}}\|u\|_{1,2}^{\frac{(1-\alpha)d}{2}} \\
  &\le \delta_1\|v\|_{1,4}^4 + \delta_2\|u\|_{1,2}^2 + \tilde{c}(\delta_1,\delta_2,\alpha,\|u\|_2) \\
  &= \delta_1\int_\Omega u^2|\nabla\log u|^4\,\d{x} + \delta_2\int_\Omega |\nabla u|^2\,\d{x} + c(\delta_1,\delta_2,\alpha,\|u\|_2),
 \end{align*}
 which is the required inequality.\quad\endproof

\section{Proof of Proposition~\ref{prop::dummy}}\label{appendix::dummy}
 To show the assertion, we employ the method introduced in \cite{matthes} (c.f.~\cite{matthes2}). We introduce the functions
 \[
  \xi = \frac{|\nabla\rho|}{\rho},\quad \eta=\frac{1}{d}\frac{\Delta\rho}{\rho},\quad (\eta + \mu)\xi^2 = \frac{1}{\rho^3}\nabla^2\rho\!:\!(\nabla\rho)^2,
 \]
 where $A\!:\!(b)^2 = \sum_{i,j=1}^d a_{ij}b_ib_j$ for $A=(a_{ij})\in \rr^{d\times d}$, $b\in\rr^d$, and define $\varrho\ge 0$ by
 \[
  \left|\frac{\nabla^2\rho}{\rho}\right|^2 = d\eta^2 + \frac{d}{d-1}\mu^2 + \varrho^2,
 \]
 which exists due to \cite[Lemma~3]{matthes}. Now set the functionals
 \begin{align*}
  \mathcal{J}(\rho)&=\int_\Omega \rho^2\left[\left|\frac{\nabla^2\rho}{\rho}\right|^2 + \frac{\Delta\rho}{\rho}\left|\frac{\nabla\rho}{\rho}\right|^2 - \delta\left|\frac{\nabla\rho}{\rho}\right|^4\right]\d{x} \\
  &= \int_\Omega \rho^2\left[\left(d\eta^2 + \frac{d}{d-1}\mu^2 + \varrho^2\right) + d\eta\xi^2 - \delta\xi^4\right]\d{x}, \\
  \mathcal{K}(\rho)&= \int_\Omega |\nabla^2\rho|^2\,\d{x} = \int_\Omega \rho^2\left(d\eta^2 + \frac{d}{d-1}\mu^2 + \varrho^2\right)\d{x},
 \end{align*}
 and the dummy integral expression
 \[
  \mathcal{I}(\rho) = \int_\Omega \div\left( \rho^{-1}|\nabla\rho|^2\nabla\rho\right)\d{x} = \int_\Omega \rho^2\big((d+2)\eta\xi^2 + 2\mu\xi^2 - \xi^4\big)\,\d{x},
 \]
 which is zero due to the boundary condition $\partial_\nu \rho=0$ on $\Gamma$. The objective is to find constants $c_0$, $c_1>0$ such that $\mathcal{J}-c_0\mathcal{K}=\mathcal{J}-c_0\mathcal{K}+c_1\mathcal{I}\ge 0$, and in such a way that $c_0$ is chosen to be as large as possible. Summing the functionals up as described, % we obtain
 \begin{equation}\label{lem::discrete::eq::functional}
  \qquad(\mathcal{J}-c_0\mathcal{K}+c_1\mathcal{I})(\rho) = \int_\Omega \rho^2\left[(1-c_0)d\eta^2 + (d+c_1(d+2))\eta\xi^2 + Q(\mu,\xi,\varrho)\right]\d{x},
 \end{equation}
 where $Q$ is the polynomial in $\mu,\xi$ and $\varrho$ given by
 \[
  Q(\mu,\xi,\varrho) = (1-c_0)\frac{d}{d-1}\mu^2 + 2c_1\mu\xi^2 - (c_1 + \delta)\xi^4 + (1-c_0)\varrho^2.
 \]
 By choosing $c_1=-d/(d+2)$, the second term on the right hand side of (\ref{lem::discrete::eq::functional}) vanishes. We write $Q(\mu,\xi,\varrho)= b_1\mu^2 + 2b_2\mu\xi^2 + b_3\xi^4 + b_4\varrho^2$, where
 \[
  b_1=\frac{(1-c_0)d}{d-1},\quad b_2=-\frac{d}{d+2},\quad b_3=\frac{d}{d+2}-\delta,\quad b_4=(1-c_0),
 \]
 and demand that $Q\ge 0$ for any given $\mu,\xi,\varrho$. If $c_0\le 1$, then $b_4\ge 0$. Now we choose $c_0\le 1$ in such a way that the remaining terms in $Q$ is nonnegative as well. This is the case if $b_1>0$ and $b_1b_3 - b_2^2\ge 0$. The second condition is equivalent to
 \[
  (1-c_0)(d+2)(1-(d+2)\delta/d) - (d-1)\ge 0,
 \]
 and therefore
 \[
  c_0 \le 1-\frac{d-1}{(d+2)(1-(d+2)\delta/d)} < 1-\frac{d-1}{d+2}=\frac{3}{d+2},
 \]
 for $\delta$ chosen sufficiently small. Choosing such a $c_0\in(0,3/(d+2))$, and using again the inequality $\|\nabla^2\rho\|_2\le \|\Delta \rho\|_2$, we obtain
 \[
  (\mathcal{J}-c_0\mathcal{K})(\rho) \ge \int_\Omega \rho^2(1-c_0)d\eta^2\,\d{x} = \frac{1-c_0}{d}\int_\Omega |\Delta\rho|^2\,\d{x} \ge \frac{1-c_0}{d}\mathcal{K}(\rho),
 \]
 which yields further, $\mathcal{J}(\rho) \ge \gamma\,\mathcal{K}(\rho)$, with $\gamma=(1 + (d-1)c_0)/d$.\quad\endproof
 
 \bibliographystyle{plain}
\bibliography{biblio} 

\end{document}